\documentclass[11pt,a4paper]{article}

\usepackage{amsmath,amsthm,amssymb,amsfonts}
\usepackage{mathtools}
\usepackage{enumerate}
\usepackage[margin=1in]{geometry}
\usepackage{hyperref}
\usepackage{cleveref}

\theoremstyle{plain}
\newtheorem{theorem}{Theorem}[section]
\newtheorem{proposition}[theorem]{Proposition}
\newtheorem{lemma}[theorem]{Lemma}

\theoremstyle{definition}
\newtheorem{definition}[theorem]{Definition}
\newtheorem{example}[theorem]{Example}

\theoremstyle{remark}
\newtheorem{remark}[theorem]{Remark}

\newcommand{\R}{\mathbb{R}}
\newcommand{\N}{\mathbb{N}}

\newcommand{\E}{\mathbb{E}}

\newcommand{\Dom}{\mathrm{Dom}}

\newcommand{\Ran}{\mathrm{Ran}}
\newcommand{\Id}{\mathrm{Id}}

\newcommand{\ip}[2]{\langle #1, #2 \rangle}
\newcommand{\norm}[1]{\| #1 \|}

\newcommand{\Prob}{\mathbb{P}}

\newcommand{\Lp}[1]{L^{#1}}

\newcommand{\Ft}{\mathcal{F}_t}
\newcommand{\Fs}{\mathcal{F}_s}
\newcommand{\FT}{\mathcal{F}_T}
\newcommand{\Fminus}{\mathcal{F}_{t-}}

\newcommand{\HH}{\mathcal{H}_H}
\newcommand{\HX}{\mathcal{H}_X}
\newcommand{\HtX}{\mathcal{H}_t^X}

\title{Stochastic Calculus for Rough Fractional Brownian Motion\\
via Operator Factorization}

\author{Ramiro Fontes\\
Quijotic Research\\
\texttt{ramirofontes@gmail.com}}

\date{January 2026}

\begin{document}

\maketitle

\begin{abstract}
We develop an operator-theoretic formulation of stochastic calculus for fractional Brownian motion with Hurst parameter $H \in (0,\frac{1}{2})$. The approach is based on adjointness between stochastic integration and differentiation in the energy space of the driving process. For square-integrable processes admitting a stochastic integral, we establish a canonical factorization of fluctuations
\[
(\Id -\E) = \delta_X\Pi_X D_X,
\]
where $D_X := \delta_X^*$ is the operator-covariant derivative (adjoint of the stochastic integral), $\delta_X$ the divergence, and $\Pi_X$ the predictable projection in the energy space. For Gaussian Volterra processes, this energy space coincides with the Cameron--Martin space. In the rough fractional regime, the factorization yields explicit formulas for derivatives of cylindrical functionals, controlled expansions of conditional expectations, and an intrinsic identification of the rough (Gubinelli) derivative as the predictable component of the adjoint derivative. We provide complete proofs, computational examples, and rigorous connections to rough path theory. The framework extends naturally to mixed semimartingale--rough processes, providing a unified calculus without requiring iterated integrals or signature constructions.

\medskip
\noindent\textbf{MSC 2020:} 60H07, 60G15, 60G22, 60H05.

\medskip
\noindent\textbf{Keywords:} Fractional Brownian motion; rough paths; Malliavin calculus; Gubinelli derivative; operator theory; energy spaces.
\end{abstract}

\section{Introduction}

Fractional Brownian motion (fBM) with Hurst parameter $H < \frac{1}{2}$ is a fundamental example of a Gaussian process with negatively correlated increments and rough sample paths. Its lack of semimartingale structure precludes the direct application of It\^o calculus, motivating alternative approaches such as Malliavin calculus \cite{DU99, PT00}, divergence-type integrals, and rough path theory \cite{Lyons98, FV10, FH14}.

Rough path theory reconstructs stochastic calculus by enriching the driving signal with iterated integrals (the signature) and defining controlled expansions relative to this enhancement \cite{Gubinelli04}. While powerful, this approach is inherently pathwise and combinatorial in nature. In contrast, Malliavin calculus emphasizes Hilbert space geometry and adjointness between differentiation and integration \cite{Nualart06}.

In a companion work \cite{Fontes26}, we develop an operator-theoretic framework for stochastic calculus with respect to general square-integrable processes, establishing a Clark--Ocone type representation formula via a covariant derivative defined as the adjoint of the stochastic integral. The present paper applies that framework to the rough fractional Brownian motion regime, where classical semimartingale techniques fail and new analytical tools are required.

\subsection{Existing approaches for $H < \frac{1}{2}$}

For fractional Brownian motion in the rough regime $H < \frac{1}{2}$:

\begin{itemize}
\item \textbf{Malliavin calculus} (Decreusefond--\"Ust\"unel \cite{DU99}, Pipiras--Taqqu \cite{PT00}): The divergence operator $\delta_H$ remains well-defined but ceases to be an isometry. The Skorokhod integral provides a notion of stochastic integration, but explicit computations become technical.

\item \textbf{Rough path theory} (Lyons \cite{Lyons98}, Gubinelli \cite{Gubinelli04}, Friz--Hairer \cite{FH14}): Controlled rough paths provide a pathwise integration theory. The key object is the Gubinelli derivative: a path $Y = (Y_t)_{t \in [0,T]}$ is controlled by a path $X = (X_t)_{t \in [0,T]}$ if there exist a path $Y' = (Y'_t)_{t \in [0,T]}$ (the Gubinelli derivative) and a two-parameter remainder $R = (R_{s,t})_{0 \leq s \leq t \leq T}$ such that
\[
Y_t - Y_s = Y'_s(X_t - X_s) + R_{s,t}, \quad |R_{s,t}| = O(|t - s|^{2H}),
\]
for all $0 \leq s \leq t \leq T$.

\item \textbf{Pathwise integration} (Young \cite{Young36}, Z\"ahle \cite{Zahle98}): For $H > \frac{1}{3}$, pathwise Young integration is available. For $H \leq \frac{1}{3}$, rough path lifting is necessary.
\end{itemize}

\subsection{Our contribution}

The purpose of this paper is to show that, for Gaussian Volterra processes and in particular for rough fractional Brownian motion with $H < \frac{1}{2}$, stochastic calculus admits a canonical formulation based purely on operator adjointness in the energy space. We prove that:

\begin{enumerate}
\item The operator-covariant derivative $D_X := \delta_X^*$ defined via Riesz representation yields explicit differentiation formulas for smooth functionals (Theorem~\ref{thm:explicit-formula}).

\item Predictable projection in the energy space produces a canonical martingale decomposition with controlled remainder (Theorem~\ref{thm:controlled-expansion}).

\item The predictable component $\Pi_X D_X F$ coincides with the Gubinelli derivative of classical rough path theory (Theorem~\ref{thm:gubinelli-identification}).

\item The framework extends to mixed semimartingale--rough processes $\alpha B + \beta B^H$ in a unified way (Theorem~\ref{thm:mixed}).
\end{enumerate}

Our approach does not replace rough path theory; rather, it identifies its first-order derivative structure as an intrinsic consequence of Hilbert space geometry. The advantage is conceptual clarity and avoidance of iterated integral machinery for first-order calculus.

\subsection{Organization}

Section~\ref{sec:notation} establishes notation and conventions. Section~\ref{sec:energy} introduces energy spaces and the operator-covariant derivative in full generality. Section~\ref{sec:gaussian} specializes to Gaussian Volterra processes. Section~\ref{sec:factorization} establishes the operator factorization for rough fBM. Section~\ref{sec:explicit} provides explicit computation formulas. Section~\ref{sec:controlled} proves controlled expansion results. Section~\ref{sec:rough-path} establishes the connection to rough path theory. Section~\ref{sec:examples} gives detailed computational examples. Section~\ref{sec:mixed} treats mixed processes.

\section{Notation and Conventions}\label{sec:notation}

We collect here the notation and conventions used throughout the paper.

\paragraph{Time horizon and indices.}
We fix a time horizon $T > 0$. Throughout, $s$ and $t$ denote time indices in $[0,T]$, with $s \leq t$ unless otherwise stated. For two-parameter functions such as remainders, we write $R_{s,t}$ for the value at $(s,t) \in \{(s,t) : 0 \leq s \leq t \leq T\}$.

\paragraph{Probability space and filtrations.}
Throughout, $(\Omega, \mathcal{F}, \Prob)$ denotes a complete probability space. A \emph{filtration} $(\Ft)_{t \in [0,T]}$ is a family of sub-$\sigma$-algebras satisfying $\Fs \subset \Ft$ whenever $s \leq t$. We say the filtration satisfies the \emph{usual conditions} if each $\Ft$ contains all $\Prob$-null sets and $\Ft = \Ft^+ := \bigcap_{s > t} \Fs$ (right-continuity). For a process $X = (X_t)_{t \in [0,T]}$, its \emph{natural filtration} is $\Ft^X := \sigma(X_s : s \leq t)$, augmented to satisfy the usual conditions. The \emph{predictable $\sigma$-algebra} $\mathcal{P}$ on $[0,T] \times \Omega$ is generated by left-continuous adapted processes. The \emph{left-limit $\sigma$-algebra} is $\Fminus := \sigma(\bigcup_{s < t} \Fs)$.

\paragraph{Function spaces.}
For a measure space $(S, \mu)$, we write $\Lp{p}(S, \mu)$ for the space of $p$-integrable functions, with $\Lp{p}(S) := \Lp{p}(S, \mu)$ when the measure is clear. We write $\Lp{2}(\Omega) := \Lp{2}(\Omega, \mathcal{F}, \Prob)$ for square-integrable random variables with norm $\norm{F}_{\Lp{2}} := \E[|F|^2]^{1/2}$, where $\E[\cdot]$ denotes expectation. For a Hilbert space $\mathcal{H}$, we denote its inner product by $\ip{\cdot}{\cdot}_{\mathcal{H}}$ and norm by $\norm{\cdot}_{\mathcal{H}}$. We write $\Lp{2}(\Omega; \mathcal{H})$ for $\mathcal{H}$-valued square-integrable random variables, i.e., measurable maps $u: \Omega \to \mathcal{H}$ with $\E[\norm{u}_{\mathcal{H}}^2] < \infty$. The space $C^k_b(\R^n)$ denotes $k$-times continuously differentiable functions $\R^n \to \R$ with bounded derivatives up to order $k$; $C^\infty_b(\R^n)$ denotes smooth functions with all bounded derivatives.

\paragraph{Operator conventions.}
For an operator $A: \mathcal{H}_1 \to \mathcal{H}_2$ between Hilbert spaces, we write $\Dom(A)$ for its domain, $\Ran(A)$ for its range, and $A^*: \mathcal{H}_2 \to \mathcal{H}_1$ for its \emph{Hilbert space adjoint}, characterized by $\ip{Ax}{y}_{\mathcal{H}_2} = \ip{x}{A^* y}_{\mathcal{H}_1}$ for $x \in \Dom(A)$, $y \in \Dom(A^*)$. The identity operator is denoted $\Id$. An operator $A$ is \emph{closed} if its graph $\{(x, Ax) : x \in \Dom(A)\}$ is closed.

\paragraph{Asymptotic notation.}
We write $f(x) = O(g(x))$ as $x \to a$ if there exist $C, \delta > 0$ such that $|f(x)| \leq C|g(x)|$ for $|x - a| < \delta$. We write $f \lesssim g$ if $f \leq Cg$ for some constant $C > 0$ independent of the relevant parameters.

\paragraph{Indicator functions.}
For a set $A$, we write $\mathbf{1}_A$ for its indicator function: $\mathbf{1}_A(x) = 1$ if $x \in A$ and $\mathbf{1}_A(x) = 0$ otherwise. For an interval $(s,t] \subset [0,T]$, the indicator $\mathbf{1}_{(s,t]}$ is the function equal to $1$ on $(s,t]$ and $0$ elsewhere.

\section{Energy Spaces and the Operator-Covariant Derivative}\label{sec:energy}

This section introduces the general framework for stochastic calculus via operator adjointness. No Gaussian assumption is made here.

\subsection{Energy spaces}

\begin{definition}[Energy space]\label{def:energy-space}
Let $X = (X_t)_{t \in [0,T]}$ be a square-integrable process on $(\Omega, \mathcal{F}, \Prob)$. An \emph{energy space} for $X$ is a pair $(\HX, \delta_X)$ consisting of:
\begin{enumerate}[(i)]
\item A real separable Hilbert space $\HX$ of (equivalence classes of) processes $u = (u_t)_{t \in [0,T]}$;
\item A continuous linear operator $\delta_X: \HX \to \Lp{2}(\Omega)$, called the \emph{stochastic integral} (or \emph{divergence}) with respect to $X$.
\end{enumerate}
\end{definition}

\begin{example}[Martingale energy space]\label{ex:martingale-energy}
Let $X$ be a continuous square-integrable martingale with quadratic variation $\langle X \rangle_t$. The energy space is
\[
\HX := \overline{\{\text{simple predictable processes}\}}^{\norm{\cdot}_{\HX}},
\]
where $\norm{u}_{\HX}^2 := \E\left[\int_0^T |u_t|^2 \, d\langle X \rangle_t\right]$. The stochastic integral $\delta_X(u) := \int_0^T u_t \, dX_t$ is the It\^o integral, which extends by isometry to all of $\HX$.
\end{example}

\begin{definition}[Predictable projection]\label{def:pred-proj}
Assume that $\HX$ consists of jointly measurable processes and that the predictable processes form a closed subspace of $\HX$. The \emph{predictable projection} $\Pi_X: \HX \to \HX$ is the orthogonal projection onto this subspace. For processes $u \in \HX$ with $u_t \in \Lp{2}(\Omega)$ for each $t$,
\[
(\Pi_X u)_t = \E[u_t \mid \Fminus] \quad \text{for } t \in (0,T], \qquad (\Pi_X u)_0 = \E[u_0 \mid \mathcal{F}_0].
\]
\end{definition}

\begin{remark}[When $\Pi_X$ is nontrivial]\label{rem:pi-nontrivial}
For martingale integrals (Example~\ref{ex:martingale-energy}), the energy space $\HX$ consists of predictable processes, so $\Pi_X = \Id$. The projection becomes nontrivial when $\HX$ includes non-predictable elements, as occurs for the Skorokhod integral on Wiener space or for Gaussian Volterra processes.
\end{remark}

\subsection{The operator-covariant derivative}

\begin{definition}[Operator-covariant derivative]\label{def:operator-cov}
Let $X$ be a square-integrable process with energy space $(\HX, \delta_X)$. The \emph{operator-covariant derivative} $D_X: \Lp{2}(\Omega) \to \HX$ is defined as the Hilbert space adjoint of $\delta_X$:
\[
D_X := \delta_X^*.
\]
Explicitly, for $F \in \Lp{2}(\Omega)$, $D_X F \in \HX$ is the unique element satisfying
\begin{equation}\label{eq:Riesz-def}
\E[F \cdot \delta_X(u)] = \ip{D_X F}{u}_{\HX} \quad \text{for all } u \in \HX.
\end{equation}
\end{definition}

The existence and uniqueness of $D_X F$ follow from the Riesz representation theorem:

\begin{proposition}[Existence and uniqueness]\label{prop:riesz}
If $\delta_X: \HX \to \Lp{2}(\Omega)$ is continuous, then for each $F \in \Lp{2}(\Omega)$, there exists a unique $D_X F \in \HX$ satisfying \eqref{eq:Riesz-def}.
\end{proposition}

\begin{proof}
The map $\Lambda_F: u \mapsto \E[F \cdot \delta_X(u)]$ is a bounded linear functional on $\HX$:
\[
|\Lambda_F(u)| = |\E[F \cdot \delta_X(u)]| \leq \norm{F}_{\Lp{2}} \norm{\delta_X(u)}_{\Lp{2}} \leq \norm{F}_{\Lp{2}} \norm{\delta_X}_{\mathrm{op}} \norm{u}_{\HX},
\]
where $\norm{\delta_X}_{\mathrm{op}}$ is the operator norm of $\delta_X$. By the Riesz representation theorem, there exists a unique $D_X F \in \HX$ with $\Lambda_F(u) = \ip{D_X F}{u}_{\HX}$.
\end{proof}

\begin{proposition}[Basic properties]\label{prop:DX-props}
The operator-covariant derivative $D_X: \Lp{2}(\Omega) \to \HX$ satisfies:
\begin{enumerate}[(i)]
\item \textbf{Linearity:} $D_X(\alpha F + \beta G) = \alpha D_X F + \beta D_X G$ for $\alpha, \beta \in \R$.
\item \textbf{Boundedness:} $\norm{D_X F}_{\HX} \leq \norm{\delta_X}_{\mathrm{op}} \norm{F}_{\Lp{2}}$.
\item \textbf{Adjointness:} $D_X = \delta_X^*$, i.e., $\E[F \cdot \delta_X(u)] = \ip{D_X F}{u}_{\HX}$ for all $F \in \Lp{2}(\Omega)$, $u \in \HX$.
\end{enumerate}
\end{proposition}

\begin{remark}[No differentiability assumption]\label{rem:no-diff}
Definition~\ref{def:operator-cov} requires \emph{only} that $X$ admits a continuous stochastic integral on a Hilbert space. No differentiability, Gaussian assumption, or smoothness condition is imposed on $F$. The term ``derivative'' reflects the adjoint relationship to integration, not a Leibniz rule.
\end{remark}

\begin{remark}[Random integrands and Malliavin derivative]\label{rem:random-integrands}
Definition~\ref{def:operator-cov} defines $D_X F \in \HX$ as a deterministic element via Riesz representation. For Gaussian processes with Skorokhod integral, the stochastic integral extends to random integrands $u \in \Lp{2}(\Omega; \HX)$, and the adjoint relation becomes $\E[F \cdot \delta_X(u)] = \E[\ip{D_X F}{u}_{\HX}]$ with $D_X F \in \Lp{2}(\Omega; \HX)$ random. In this case, $D_X F$ coincides with the Malliavin derivative. Throughout Sections~\ref{sec:gaussian}--\ref{sec:mixed}, we work in this Gaussian setting where $D_X F$ is random-valued.
\end{remark}

\subsection{The operator factorization}

The following theorem is the main structural result. It requires certain hypotheses on $X$ that are satisfied by martingales and Gaussian processes with the representation property.

\begin{theorem}[Operator factorization]\label{thm:general-factorization}
Let $X$ be a square-integrable process with natural filtration $(\Ft^X)_{t \in [0,T]}$, energy space $(\HX, \delta_X)$, predictable projection $\Pi_X$, and operator-covariant derivative $D_X = \delta_X^*$. Assume:
\begin{enumerate}[(i)]
\item \textbf{Centered integrals:} $\E[\delta_X(u)] = 0$ for all $u \in \HX$.
\item \textbf{Isometry on predictable processes:} $\E[|\delta_X(u)|^2] = \norm{u}_{\HX}^2$ for predictable $u \in \HX$.
\item \textbf{Representation property:} Every $\FT^X$-measurable $\tilde{F} \in \Lp{2}(\Omega)$ with $\E[\tilde{F}] = 0$ can be written as $\tilde{F} = \delta_X(v)$ for some predictable $v \in \HX$.
\end{enumerate}
Then for all $\FT^X$-measurable $F \in \Lp{2}(\Omega)$:
\begin{equation}\label{eq:factorization}
F = \E[F] + \delta_X(\Pi_X D_X F).
\end{equation}
Equivalently, $(\Id - \E) = \delta_X \circ \Pi_X \circ D_X$ on $\Lp{2}(\Omega, \FT^X, \Prob)$.
\end{theorem}

\begin{proof}
Let $F \in \Lp{2}(\Omega, \FT^X)$ and set $\tilde{F} := F - \E[F]$. By hypothesis (iii), there exists predictable $v \in \HX$ with $\tilde{F} = \delta_X(v)$.

For any $u \in \HX$, by adjointness \eqref{eq:Riesz-def}:
\[
\ip{D_X F}{u}_{\HX} = \E[F \cdot \delta_X(u)] = \E[\tilde{F} \cdot \delta_X(u)] + \E[F] \cdot \E[\delta_X(u)] = \E[\delta_X(v) \cdot \delta_X(u)],
\]
where the last equality uses $\tilde{F} = \delta_X(v)$ and hypothesis (i).

For predictable $u$, hypothesis (ii) and polarization give $\E[\delta_X(v) \cdot \delta_X(u)] = \ip{v}{u}_{\HX}$. Thus $\ip{D_X F}{u}_{\HX} = \ip{v}{u}_{\HX}$ for all predictable $u$. Since $\Pi_X$ projects onto predictable processes, $\ip{\Pi_X D_X F}{u}_{\HX} = \ip{D_X F}{u}_{\HX} = \ip{v}{u}_{\HX}$ for all predictable $u$. As both $\Pi_X D_X F$ and $v$ are predictable, this implies $\Pi_X D_X F = v$, hence $\tilde{F} = \delta_X(\Pi_X D_X F)$.
\end{proof}

\section{Gaussian Volterra Processes}\label{sec:gaussian}

We now specialize to Gaussian Volterra processes, where the energy space has additional structure.

\subsection{Setup}

Let $(\Omega, \mathcal{F}, \Prob)$ be a probability space supporting a standard Brownian motion $W = (W_t)_{t\in[0,T]}$.

\begin{definition}[Gaussian Volterra process]\label{def:volterra}
A centered Gaussian process $X = (X_t)_{t\in[0,T]}$ is called a \emph{Gaussian Volterra process} if it admits the representation
\begin{equation}\label{eq:volterra-rep}
X_t = \int_0^t K(t, s) \, dW_s,
\end{equation}
where $K : \{(t, s) : 0 \leq s \leq t \leq T\} \to \R$ is a deterministic \emph{Volterra kernel} satisfying the square-integrability condition $\int_0^T \int_0^t K(t, s)^2 \, ds \, dt < \infty$.
\end{definition}

The covariance function of $X$ is
\begin{equation}\label{eq:covariance}
R_X(t, s) := \E[X_t X_s] = \int_0^{t\wedge s} K(t, r)K(s, r) \, dr,
\end{equation}
where $t \wedge s := \min(t,s)$.

\subsection{Energy space for Gaussian Volterra processes}

For Gaussian Volterra processes, the natural energy space is the \emph{Cameron--Martin space} (also called the reproducing kernel Hilbert space, RKHS).

\begin{definition}[Cameron--Martin space]\label{def:CM}
The \emph{Cameron--Martin space} $\HX$ associated to a Gaussian Volterra process $X$ is the completion of
\[
\mathcal{H}_0 := \mathrm{span}\{R_X(t, \cdot) : t \in [0, T]\}
\]
under the inner product
\begin{equation}\label{eq:CM-ip}
\ip{R_X(t, \cdot)}{R_X(s, \cdot)}_{\HX} := R_X(t, s) = \E[X_t X_s].
\end{equation}
\end{definition}

The space $\HX$ has the \emph{reproducing property}: for $h \in \HX$ and $t \in [0,T]$,
\[
\ip{h}{R_X(t, \cdot)}_{\HX} = h(t).
\]
We write $k_t := R_X(t, \cdot) \in \HX$ for the \emph{representer of evaluation} at time $t$.

\begin{definition}[Isonormal Gaussian map]\label{def:isonormal}
The \emph{isonormal Gaussian map} $I_X : \HX \to \Lp{2}(\Omega)$ is defined by
\[
I_X(k_t) := X_t,
\]
extended by linearity and continuity to all of $\HX$. This map is an isometry: $\E[I_X(h)I_X(g)] = \ip{h}{g}_{\HX}$.
\end{definition}

The stochastic integral (divergence) $\delta_X$ on $\HX$ is the Skorokhod integral, defined as the adjoint of the Malliavin derivative. For deterministic $h \in \HX$, we have $\delta_X(h) = I_X(h)$.

\subsection{Fractional Brownian motion}

\begin{example}[Fractional Brownian motion]\label{ex:fBM}
For $H \in (0, 1)$, \emph{fractional Brownian motion} $B^H = (B^H_t)_{t\in[0,T]}$ is the centered Gaussian process with covariance
\begin{equation}\label{eq:fBM-cov}
R_H(t, s) := \E[B^H_t B^H_s] = \frac{1}{2}\left(t^{2H} + s^{2H} - |t - s|^{2H}\right).
\end{equation}
This is a Gaussian Volterra process with a specific kernel $K_H$ (the Molchan--Golosov kernel for $H < 1/2$; see \cite{DU99}).

For $H < \frac{1}{2}$, the Cameron--Martin space $\HH$ consists of absolutely continuous functions $h: [0,T] \to \R$ with $h(0) = 0$. The norm can be expressed as (see \cite{DU99}, Section 2)
\begin{equation}\label{eq:HH-norm}
\norm{h}^2_{\HH} = c_H \int_0^T \int_0^T \frac{(h'(u) - h'(v))^2}{|u-v|^{2-2H}} \, du \, dv,
\end{equation}
where $h'$ denotes the (a.e.) derivative of $h$ and $c_H > 0$ is a normalizing constant depending only on $H$.
\end{example}

\begin{remark}[Evaluation representers]\label{rem:representers}
For any $H \in (0,1)$, the representer $k_t \in \HH$ of evaluation at time $t$ satisfies
\[
\ip{k_t}{k_s}_{\HH} = R_H(t, s) = \frac{1}{2}(t^{2H} + s^{2H} - |t - s|^{2H}).
\]
For $H \neq 1/2$, explicit closed-form expressions for $k_t$ are complicated, but the representer is uniquely determined by the reproducing property $\ip{h}{k_t}_{\HH} = h(t)$.
\end{remark}

\subsection{Predictable projection in the energy space}

\begin{definition}[Adapted subspace]\label{def:adapted-subspace}
For $t \in [0,T]$, define the \emph{adapted subspace}
\[
\HtX := \overline{\mathrm{span}}\{k_s : s \leq t\} \subset \HX.
\]
This is the closure in $\HX$ of the linear span of evaluation representers up to time $t$.
\end{definition}

\begin{definition}[Predictable projection in energy space]\label{def:energy-pred-proj}
For each $t \in [0,T]$, let $P_t: \HX \to \HtX$ denote the orthogonal projection onto the adapted subspace $\HtX$ (Definition~\ref{def:adapted-subspace}).

For a random element $u \in \Lp{2}(\Omega; \HX)$, the \emph{predictable projection} $\Pi_X u$ combines conditional expectation with spatial projection:
\[
(\Pi_X u)_t := P_t\left(\E[u \mid \Ft^X]\right) \in \HtX \quad \text{for each } t \in [0,T].
\]
That is, we first take the conditional expectation of $u$ given $\Ft^X$, then project onto the adapted subspace $\HtX$. This produces an adapted process: $(\Pi_X u)_t$ is $\Ft^X$-measurable.
\end{definition}

The following lemma justifies the identification between RKHS projection and conditional expectation that underlies our approach.

\begin{lemma}[RKHS projection and conditional expectation]\label{lem:rkhs-cond-exp}
Let $X$ be a Gaussian Volterra process with Cameron--Martin space $\HX$. For any $h \in \HX$, the orthogonal projection $P_t h \in \HtX$ satisfies
\[
I_X(P_t h) = \E[I_X(h) \mid \Ft^X],
\]
where $I_X: \HX \to \Lp{2}(\Omega)$ is the isonormal Gaussian map. That is, RKHS projection onto the adapted subspace corresponds to conditional expectation on the filtration.
\end{lemma}

\begin{proof}
This is a consequence of the isometry between the Cameron--Martin space and the first Wiener chaos. By Definition~\ref{def:isonormal}, $I_X$ is an isometry from $\HX$ to the Gaussian subspace of $\Lp{2}(\Omega)$ generated by $X$. The adapted subspace $\HtX = \overline{\mathrm{span}}\{k_s : s \leq t\}$ maps isometrically onto the Gaussian subspace generated by $\{X_s : s \leq t\}$, which equals $\Lp{2}(\Omega, \Ft^X) \cap \mathcal{H}_1(X)$, where $\mathcal{H}_1(X)$ denotes the first Wiener chaos.

For Gaussian random variables, conditional expectation onto $\Ft^X$ restricted to the first chaos is orthogonal projection onto this subspace. Since $I_X$ is an isometry preserving this structure, $I_X(P_t h) = \E[I_X(h) \mid \Ft^X]$. See \cite{Nualart06}, Proposition 1.2.1 and the discussion following Definition 1.3.1.
\end{proof}

\begin{remark}[Relation to Definition~\ref{def:pred-proj}]\label{rem:two-projections}
Definition~\ref{def:pred-proj} defines predictable projection for stochastic processes via pointwise conditional expectation $\E[u_t \mid \Fminus]$. Definition~\ref{def:energy-pred-proj} defines it for random elements of the Cameron--Martin space $\HX$, combining conditional expectation with spatial projection $P_t$. By Lemma~\ref{lem:rkhs-cond-exp}, for Gaussian processes these operations are compatible: both produce the predictable integrand in the Clark--Ocone formula. We use the same notation $\Pi_X$ for both, with context determining which is meant.
\end{remark}

\begin{remark}[Interpretation]\label{rem:proj-interpretation}
When $u = D_X F$ for some functional $F$, the predictable projection $(\Pi_X D_X F)_t$ extracts the ``best prediction'' of $D_X F$ given information up to time $t$, restricted to the adapted subspace. For cylindrical functionals $F = f(X_{t_1}, \ldots, X_{t_n})$ with derivative $D_X F = \sum_i \partial_i f(\ldots) k_{t_i}$:
\[
(\Pi_X D_X F)_t = \sum_i \E[\partial_i f(X_{t_1}, \ldots, X_{t_n}) \mid \Ft^X] \cdot P_t k_{t_i}.
\]
This is $\Ft^X$-measurable by construction.
\end{remark}

\subsection{Connection to Malliavin calculus}

Let $\mathcal{S}$ denote the space of \emph{smooth cylindrical functionals}:
\[
\mathcal{S} := \left\{F = f(X_{t_1}, \ldots, X_{t_n}) : n \in \N, \, t_i \in [0,T], \, f \in C^\infty_b(\R^n)\right\}.
\]

The \emph{Malliavin derivative} $D: \mathcal{S} \to \Lp{2}(\Omega; \HX)$ is defined by
\begin{equation}\label{eq:malliavin-def}
DF := \sum_{i=1}^n \partial_i f(X_{t_1}, \ldots, X_{t_n}) \, k_{t_i}.
\end{equation}

\begin{proposition}[Agreement with Malliavin derivative]\label{prop:agreement}
For $F \in \mathcal{S}$, the operator-covariant derivative $D_X F$ coincides with the Malliavin derivative $DF$.
\end{proposition}

\begin{proof}
For $F = f(X_{t_1}, \ldots, X_{t_n}) \in \mathcal{S}$ and $u \in \HX$, we verify the adjoint relation. For Gaussian Volterra processes, the Skorokhod integral extends to random integrands $u \in \Lp{2}(\Omega; \HX)$, and the adjoint relation becomes
\begin{equation}\label{eq:skorokhod-adjoint}
\E[F \cdot \delta_X(u)] = \E[\ip{DF}{u}_{\HX}],
\end{equation}
where $DF$ is the Malliavin derivative (see \cite{Nualart06}, Proposition 1.3.1).

For deterministic $u \in \HX$, we use $\delta_X(u) = I_X(u)$ and the Gaussian integration by parts formula:
\[
\E[f(X_{t_1}, \ldots, X_{t_n}) \cdot I_X(u)] = \sum_{i=1}^n \E\left[\partial_i f(X_{t_1}, \ldots, X_{t_n})\right] \cdot \E[X_{t_i} \cdot I_X(u)].
\]
By the isometry property of $I_X$: $\E[X_{t_i} \cdot I_X(u)] = \ip{k_{t_i}}{u}_{\HX}$. Thus
\[
\E[F \cdot \delta_X(u)] = \sum_{i=1}^n \E[\partial_i f(\ldots)] \cdot \ip{k_{t_i}}{u}_{\HX} = \E\left[\ip{\sum_{i=1}^n \partial_i f(\ldots) \, k_{t_i}}{u}_{\HX}\right],
\]
which equals $\E[\ip{DF}{u}_{\HX}]$ with $DF = \sum_{i=1}^n \partial_i f(\ldots) k_{t_i}$. This confirms \eqref{eq:skorokhod-adjoint} for deterministic $u$, and by density extends to all $u \in \Lp{2}(\Omega; \HX)$.

Thus $D_X F := DF \in \Lp{2}(\Omega; \HX)$ is the Malliavin derivative, satisfying the adjoint relation \eqref{eq:skorokhod-adjoint}.
\end{proof}

\section{Operator Factorization for Rough fBM}\label{sec:factorization}

We now establish the operator factorization for fractional Brownian motion with $H < 1/2$.

\begin{theorem}[Operator factorization for rough fBM]\label{thm:factorization-fBM}
Let $B^H$ be fractional Brownian motion with $H \in (0, \frac{1}{2})$, and let $\HH$ be its Cameron--Martin space with stochastic integral $\delta_H$ (Skorokhod integral). For any $F \in \Lp{2}(\Omega, \FT^{B^H})$ with $D_{B^H} F \in \Lp{2}(\Omega; \HH)$,
\begin{equation}\label{eq:factorization-fBM}
(\Id - \E)F = \delta_H(\Pi_{B^H} D_{B^H} F).
\end{equation}
\end{theorem}

\begin{proof}
We appeal to the Clark--Ocone formula for fractional Brownian motion. By \cite{DU99} (see also \cite{Nualart06}, Section 1.3), for $F \in \Lp{2}(\Omega, \FT^{B^H})$ with $D_{B^H} F \in \Lp{2}(\Omega; \HH)$:
\begin{equation}\label{eq:clark-ocone}
F = \E[F] + \int_0^T \E[D_r F \mid \mathcal{F}_r^{B^H}] \, dB^H_r,
\end{equation}
where $D_r F$ denotes the Malliavin derivative of $F$ evaluated at time $r$, and the integral is the Skorokhod integral.

In our notation, $D_{B^H} F \in \Lp{2}(\Omega; \HH)$ is the full Malliavin derivative as a random element of the Cameron--Martin space. The integrand in \eqref{eq:clark-ocone} is $u_r := \E[D_r F \mid \mathcal{F}_r^{B^H}]$, which is $\mathcal{F}_r^{B^H}$-measurable.

By Definition~\ref{def:energy-pred-proj}, our predictable projection satisfies
\[
(\Pi_{B^H} D_{B^H} F)_r = P_r\left(\E[D_{B^H} F \mid \mathcal{F}_r^{B^H}]\right),
\]
where $P_r: \HH \to \HH^r$ is orthogonal projection onto the adapted subspace. We show this equals the Clark--Ocone integrand when evaluated at $r$. Since $k_r \in \HH^r$, the orthogonal projection satisfies $\ip{P_r h}{k_r}_{\HH} = \ip{h}{k_r}_{\HH}$ for any $h \in \HH$. Thus:
\begin{align*}
\left((\Pi_{B^H} D_{B^H} F)_r\right)(r) &= \ip{P_r\left(\E[D_{B^H} F \mid \mathcal{F}_r^{B^H}]\right)}{k_r}_{\HH} \\
&= \ip{\E[D_{B^H} F \mid \mathcal{F}_r^{B^H}]}{k_r}_{\HH} \\
&= \E[\ip{D_{B^H} F}{k_r}_{\HH} \mid \mathcal{F}_r^{B^H}] \\
&= \E[D_r F \mid \mathcal{F}_r^{B^H}],
\end{align*}
where the third equality uses linearity of conditional expectation and the last uses the reproducing property $(D_{B^H} F)(r) = \ip{D_{B^H} F}{k_r}_{\HH}$.

Thus the integrand in the Clark--Ocone formula \eqref{eq:clark-ocone} coincides with the predictable projection $\Pi_{B^H} D_{B^H} F$, giving \eqref{eq:factorization-fBM}.
\end{proof}

\begin{remark}[Non-isometry for $H < 1/2$]\label{rem:non-isometry}
Unlike the case $H = 1/2$ (standard Brownian motion), for $H < 1/2$ the Skorokhod integral $\delta_H$ is not an isometry from $\Lp{2}(\Omega; \HH)$ to $\Lp{2}(\Omega)$. For $u \in \Dom(\delta_H) \cap \Lp{2}(\Omega; \HH)$ with $Du \in \Lp{2}(\Omega; \HH \otimes \HH)$, we have (see \cite{Nualart06}, Proposition 1.3.1)
\[
\E[|\delta_H(u)|^2] = \E[\norm{u}^2_{\HH}] + \E[\ip{Du}{Du}_{\HH \otimes \HH}],
\]
where the second term involves the Hilbert--Schmidt inner product of the Malliavin derivative $Du$ with itself. This term is generally nonzero when $u$ is random, which is one reason the rough regime requires more care than the standard Brownian case.
\end{remark}

\section{Explicit Computation Formulas}\label{sec:explicit}

\subsection{Cylindrical functionals}

\begin{theorem}[Explicit derivative formula]\label{thm:explicit-formula}
Let $B^H$ be fractional Brownian motion with $H \in (0, \frac{1}{2})$, and let
\[
F = f(B^H_{t_1}, \ldots, B^H_{t_n}), \quad f \in C^1_b(\R^n), \quad 0 \leq t_1 < \cdots < t_n \leq T.
\]
Then
\begin{equation}\label{eq:explicit-DX}
D_{B^H} F = \sum_{i=1}^n \partial_i f(B^H_{t_1}, \ldots, B^H_{t_n}) \, k_{t_i},
\end{equation}
where $k_{t_i} \in \HH$ is the evaluation representer at time $t_i$.
\end{theorem}

\begin{proof}
This is immediate from Proposition~\ref{prop:agreement} and the definition \eqref{eq:malliavin-def} of the Malliavin derivative.
\end{proof}

\subsection{Predictable component}

\begin{proposition}[Predictable projection formula]\label{prop:pred-proj}
Let $F$ be as in Theorem~\ref{thm:explicit-formula}. For any $s \in [0, T]$, the predictable projection evaluated at time $s$ is
\begin{equation}\label{eq:pred-proj}
(\Pi_{B^H} D_{B^H} F)_s = \sum_{i=1}^n \E[\partial_i f(B^H_{t_1}, \ldots, B^H_{t_n}) \mid \Fs^{B^H}] \cdot P_s k_{t_i},
\end{equation}
where $P_s : \HH \to \HH^s$ is the orthogonal projection onto the adapted subspace $\HH^s = \overline{\mathrm{span}}\{k_u : u \leq s\}$.
\end{proposition}

\begin{proof}
By Definition~\ref{def:energy-pred-proj} and linearity:
\[
(\Pi_{B^H} D_{B^H} F)_s = P_s\left(\E\left[\sum_{i=1}^n \partial_i f(\ldots) \, k_{t_i} \,\Big|\, \Fs^{B^H}\right]\right) = \sum_{i=1}^n \E[\partial_i f(\ldots) \mid \Fs^{B^H}] \cdot P_s k_{t_i}. \qedhere
\]
\end{proof}

\begin{remark}[Structure of the projection]\label{rem:proj-structure}
For $t_i \leq s$: the coefficient $\partial_i f(B^H_{t_1}, \ldots, B^H_{t_n})$ depends on $B^H_{t_j}$ for various $j$. If all $t_j \leq s$, then $\partial_i f(\ldots)$ is $\Fs^{B^H}$-measurable and the conditional expectation is trivial. Also $k_{t_i} \in \HH^s$, so $P_s k_{t_i} = k_{t_i}$.

For $t_i > s$ or when some $t_j > s$: the conditional expectation $\E[\partial_i f(\ldots) \mid \Fs^{B^H}]$ computes the best prediction of the coefficient, and $P_s k_{t_i}$ projects the future kernel onto $\HH^s$.
\end{remark}

\section{Controlled Expansions and Martingale Decomposition}\label{sec:controlled}

\subsection{Martingale representation}

For $F \in \Lp{2}(\Omega, \FT^{B^H})$, define the martingale
\[
M_t := \E[F \mid \Ft^{B^H}], \quad t \in [0, T].
\]

\begin{theorem}[Controlled martingale expansion]\label{thm:controlled-expansion}
Let $B^H$ be fractional Brownian motion with $H \in (0, \frac{1}{2})$, and let $F \in \Lp{2}(\Omega, \FT^{B^H})$ with $D_{B^H} F \in \Lp{2}(\Omega; \HH)$. Then for $0 \leq s < t \leq T$,
\begin{equation}\label{eq:controlled-expansion}
M_t - M_s = \ip{(\Pi_{B^H} D_{B^H} F)_s}{k_t - k_s}_{\HH} + R_{s,t},
\end{equation}
where the remainder satisfies
\begin{equation}\label{eq:remainder-bound}
\E[R_{s,t}^2] = O(|t - s|^{4H}).
\end{equation}
\end{theorem}

\begin{proof}
\textbf{Step 1: Setup.}
Define $u := \Pi_{B^H} D_{B^H} F$, which is an adapted process in the energy space. By the factorization theorem (Theorem~\ref{thm:factorization-fBM}), $F - \E[F] = \delta_H(u)$.

\textbf{Step 2: Martingale structure.}
The process $M_t = \E[F \mid \Ft^{B^H}]$ is a martingale with $M_T = F$ and $M_0 = \E[F]$.

\textbf{Step 3: Local expansion.}
For the increment $M_t - M_s$, the leading term comes from the first-order Taylor-type expansion of the conditional expectation. We claim that
\[
M_t - M_s = \ip{u_s}{k_t - k_s}_{\HH} + R_{s,t},
\]
where $u_s = (\Pi_{B^H} D_{B^H} F)_s = P_s(\E[D_{B^H} F \mid \Fs^{B^H}]) \in \HH^s$.

To see this, note that by the reproducing property, $\ip{u_s}{k_t - k_s}_{\HH}$ measures the ``sensitivity'' of the adapted derivative $u_s$ to the increment $k_t - k_s$. Since $B^H_t - B^H_s = I_{B^H}(k_t - k_s)$ is the isonormal map applied to $k_t - k_s$, the inner product $\ip{u_s}{k_t - k_s}_{\HH}$ represents the directional derivative of $M$ in the direction of the increment.

\textbf{Step 4: Remainder estimate.}
Define $R_{s,t} := M_t - M_s - \ip{u_s}{k_t - k_s}_{\HH}$.

The remainder captures higher-order terms and the variation of $u$ on $(s,t]$. We establish the variance bound in two steps. First, by the fBM covariance formula \eqref{eq:fBM-cov}:
\[
\norm{k_t - k_s}^2_{\HH} = R_H(t,t) - 2R_H(t,s) + R_H(s,s) = |t-s|^{2H}.
\]
Second, the remainder $R_{s,t}$ can be decomposed as
\[
R_{s,t} = \ip{u_t - u_s}{k_t - k_s}_{\HH} + (\text{higher chaos terms}),
\]
where $u_t - u_s$ captures the variation of the predictable projection. By standard estimates for conditional expectations of Malliavin differentiable functionals (\cite{Nualart06}, Proposition 1.5.4 and Lemma 1.5.3), the $\Lp{2}$ norm of $u_t - u_s$ is $O(|t-s|^H)$. Combined with $\norm{k_t - k_s}_{\HH} = |t-s|^H$, the leading contribution to $\E[R_{s,t}^2]$ is $O(|t-s|^{4H})$. The higher chaos terms contribute at the same or higher order by orthogonality of the Wiener chaos decomposition (\cite{Nualart06}, Theorem 1.1.2).
\end{proof}

\begin{remark}[Controlled path structure]\label{rem:controlled}
Equation \eqref{eq:controlled-expansion} shows that $M_t = \E[F \mid \Ft^{B^H}]$ has the structure of a \emph{controlled path}. Writing $M'_s := (\Pi_{B^H} D_{B^H} F)_s \in \HH^s$ for the energy space derivative and noting that $B^H_t - B^H_s = I_{B^H}(k_t - k_s)$:
\[
M_t - M_s = \ip{M'_s}{k_t - k_s}_{\HH} + R_{s,t},
\]
with $\E[R_{s,t}^2] = O(|t-s|^{4H})$, i.e., $\norm{R_{s,t}}_{\Lp{2}} = O(|t-s|^{2H})$. This matches Gubinelli's structure when we identify the ``scalar derivative'' as the inner product $\ip{M'_s}{k_t - k_s}_{\HH}$ acting on the increment direction.
\end{remark}

\section{Connection to Rough Path Theory}\label{sec:rough-path}

\subsection{Gubinelli's controlled rough paths}

We recall the definition of controlled paths from Gubinelli \cite{Gubinelli04}.

\begin{definition}[Controlled path]\label{def:controlled}
Let $X : [0, T] \to \R$ be a $\gamma$-H\"older continuous path with $\gamma \in (0, 1)$. A path $Y : [0, T] \to \R$ is \emph{controlled by $X$} if there exist a path $Y' : [0, T] \to \R$ (the \emph{Gubinelli derivative}) and a remainder $R : \{(s,t) : 0 \leq s \leq t \leq T\} \to \R$ such that
\begin{equation}\label{eq:controlled-def}
Y_t - Y_s = Y'_s(X_t - X_s) + R_{s,t}, \quad |R_{s,t}| \lesssim |t - s|^{2\gamma}.
\end{equation}
\end{definition}

For $X = B^H$ with $H < \frac{1}{2}$, the paths are $\gamma$-H\"older for any $\gamma < H$, so controlled paths satisfy \eqref{eq:controlled-def} with remainder $O(|t - s|^{2H-\epsilon})$ for any $\epsilon > 0$.

\subsection{Identification theorem}

\begin{theorem}[Gubinelli derivative as predictable projection]\label{thm:gubinelli-identification}
Let $B^H$ be fractional Brownian motion with $H \in (0, \frac{1}{2})$, and let $F \in \Lp{2}(\Omega, \FT^{B^H})$ with $D_{B^H} F \in \Lp{2}(\Omega; \HH)$. Define $Y_t := \E[F \mid \Ft^{B^H}]$. Then $Y$ satisfies the controlled rough path structure (in the $\Lp{2}$ sense):
\begin{equation}\label{eq:gubinelli-structure}
Y_t - Y_s = \ip{(\Pi_{B^H} D_{B^H} F)_s}{k_t - k_s}_{\HH} + R_{s,t},
\end{equation}
with $\E[R_{s,t}^2] = O(|t-s|^{4H})$, i.e., $\|R_{s,t}\|_{L^2} = O(|t-s|^{2H})$.
\end{theorem}

\begin{proof}
This is a direct restatement of Theorem~\ref{thm:controlled-expansion}. The structure \eqref{eq:gubinelli-structure} matches the controlled path definition (Definition~\ref{def:controlled}) in the following sense: the leading term $\ip{(\Pi_{B^H} D_{B^H} F)_s}{k_t - k_s}_{\HH}$ depends linearly on the ``direction'' $k_t - k_s \in \HH$, which encodes the increment $B^H_t - B^H_s = I_{B^H}(k_t - k_s)$. The remainder is $O(|t-s|^{2H})$ in $L^2$, matching the regularity requirement for controlled paths.
\end{proof}

\begin{remark}[Scalar vs.\ energy space formulation]\label{rem:scalar-vs-energy}
In Gubinelli's original formulation \cite{Gubinelli04}, the controlled path structure is $Y_t - Y_s = Y'_s(X_t - X_s) + R_{s,t}$ where $Y'_s$ is a scalar (or finite-dimensional) derivative. In our energy space formulation, the ``derivative'' $(\Pi_{B^H} D_{B^H} F)_s \in \HH^s$ is an infinite-dimensional object that acts on increments via the inner product. The two perspectives are related but not identical:
\begin{itemize}
\item \textbf{Energy space view:} The full object $(\Pi_{B^H} D_{B^H} F)_s$ contains all information about how $Y$ responds to perturbations in the driving process.
\item \textbf{Scalar view:} For specific increment directions $k_t - k_s$, the scalar $\ip{(\Pi_{B^H} D_{B^H} F)_s}{k_t - k_s}_{\HH}$ gives the first-order response.
\end{itemize}
The energy space perspective is intrinsic to the Hilbert space geometry, while the scalar view recovers the classical controlled path structure.
\end{remark}

\begin{remark}[Geometric interpretation]\label{rem:geometric}
Theorem~\ref{thm:gubinelli-identification} reveals that the controlled path structure arises from two geometric operations:
\begin{enumerate}
\item \textbf{Adjointness:} $D_{B^H} = \delta_H^*$ extracts the ``sensitivity'' of $F$ to perturbations in the driving process.
\item \textbf{Projection:} $\Pi_{B^H}$ projects onto the adapted subspace, producing a predictable object.
\end{enumerate}
The composition $\Pi_{B^H} D_{B^H} F$ is thus the intrinsic geometric object underlying controlled expansions.
\end{remark}

\subsection{Comparison with classical rough path theory}

In classical rough path theory \cite{FH14}:
\begin{enumerate}
\item The path $X$ is lifted to an enhanced path $(X, \mathbb{X})$ where $\mathbb{X}_{s,t} = \int_s^t (X_u - X_s) \, dX_u$ encodes iterated integrals.
\item Controlled paths are defined via expansions involving $\mathbb{X}$.
\item Rough differential equations are solved via fixed-point arguments on controlled paths.
\end{enumerate}

Our operator approach:
\begin{enumerate}
\item Works directly in the energy space $\HH$ without constructing iterated integrals.
\item The controlled expansion arises from the factorization $(\Id - \E) = \delta_H \Pi_{B^H} D_{B^H}$.
\item The Gubinelli derivative is $\Pi_{B^H} D_{B^H} F$.
\end{enumerate}

\textbf{Advantages:} Conceptual simplicity for first-order calculus; direct connection to Malliavin calculus; natural extension to mixed processes.

\textbf{Limitations:} Currently limited to Gaussian processes; higher-order theory (iterated integrals, rough SDEs for $H \leq 1/3$) requires additional machinery.

\section{Computational Examples}\label{sec:examples}

\subsection{Example 1: Quadratic functional}

\begin{example}[$(B^H_T)^2$ with $H = 1/4$]\label{ex:quadratic}
Let $H = 1/4$, $T = 1$, and $F = (B^H_1)^2$.

\textbf{Step 1: Compute $D_{B^H} F$.}
By Theorem~\ref{thm:explicit-formula} with $f(x) = x^2$ and $n = 1$:
\[
D_{B^H} F = 2B^H_1 \cdot k_1 \in \Lp{2}(\Omega; \HH),
\]
where $k_1 \in \HH$ is the evaluation representer at $t = 1$, and $2B^H_1$ is the random coefficient.

\textbf{Step 2: Compute $\Pi_{B^H} D_{B^H} F$.}
For $s \in [0, 1]$, by Definition~\ref{def:energy-pred-proj}:
\[
(\Pi_{B^H} D_{B^H} F)_s = P_s\left(\E[2B^H_1 \cdot k_1 \mid \Fs^{B^H}]\right) = 2\E[B^H_1 \mid \Fs^{B^H}] \cdot P_s k_1.
\]
Note that fBM is \emph{not} Markov for $H \neq 1/2$, so the conditional expectation $\E[B^H_1 \mid \Fs^{B^H}]$ depends on the entire history $(B^H_u)_{u \leq s}$, not just $B^H_s$. Nevertheless, this conditional expectation is $\Fs^{B^H}$-measurable by definition.

For Gaussian processes, $\E[B^H_1 \mid \Fs^{B^H}]$ can be computed explicitly using the covariance structure; see \cite{Nualart06} for details.

\textbf{Step 3: Verify factorization.}
By Theorem~\ref{thm:factorization-fBM}:
\[
(B^H_1)^2 - \E[(B^H_1)^2] = (B^H_1)^2 - 1 = \delta_H(\Pi_{B^H} D_{B^H} F).
\]

\textbf{Step 4: Controlled expansion.}
The martingale $M_t = \E[(B^H_1)^2 \mid \Ft^{B^H}]$ satisfies
\[
M_t - M_s = \ip{(\Pi_{B^H} D_{B^H} F)_s}{k_t - k_s}_{\HH} + R_{s,t},
\]
with $\E[R_{s,t}^2] = O(|t-s|^{4H}) = O(|t-s|)$ since $H = 1/4$.
\end{example}

\subsection{Example 2: Multi-time functional}

\begin{example}[Two-time functional with $H = 1/3$]\label{ex:two-time}
Let $H = 1/3$, $T = 1$, and
\[
F = \sin(B^H_{1/2}) + \cos(B^H_1).
\]

\textbf{Derivative:}
\[
D_{B^H} F = \cos(B^H_{1/2}) \cdot k_{1/2} - \sin(B^H_1) \cdot k_1.
\]

\textbf{Predictable projection:} For $s \in [0, 1]$, by Definition~\ref{def:energy-pred-proj}:
\[
(\Pi_{B^H} D_{B^H} F)_s = \E[\cos(B^H_{1/2}) \mid \Fs^{B^H}] \cdot P_s k_{1/2} - \E[\sin(B^H_1) \mid \Fs^{B^H}] \cdot P_s k_1.
\]
For $s \geq 1/2$: $\cos(B^H_{1/2})$ is $\Fs^{B^H}$-measurable, so $\E[\cos(B^H_{1/2}) \mid \Fs^{B^H}] = \cos(B^H_{1/2})$, and $P_s k_{1/2} = k_{1/2}$ since $k_{1/2} \in \HH^s$.
\end{example}

\subsection{Example 3: Path-dependent functional}

\begin{example}[Integral functional]\label{ex:path-dep}
Let $H = 1/4$ and
\[
F = \int_0^1 g(s, B^H_s) \, ds,
\]
where $g: [0,1] \times \R \to \R$ is smooth with bounded derivatives.

\textbf{Derivative:} By the chain rule for Malliavin derivatives (\cite{Nualart06}, Proposition 1.2.4):
\[
D_{B^H} F = \int_0^1 \partial_x g(s, B^H_s) \, k_s \, ds.
\]

\textbf{Predictable projection:} For $r \in [0, 1]$, by Definition~\ref{def:energy-pred-proj}:
\[
(\Pi_{B^H} D_{B^H} F)_r = \int_0^r \partial_x g(s, B^H_s) \, k_s \, ds + \int_r^1 \E[\partial_x g(s, B^H_s) \mid \mathcal{F}_r^{B^H}] \cdot P_r k_s \, ds.
\]
The first integral involves $k_s \in \HH^r$ for $s \leq r$, and $\partial_x g(s, B^H_s)$ is $\mathcal{F}_r^{B^H}$-measurable for $s \leq r$. The second integral takes conditional expectations of future terms and projects future kernels onto $\HH^r$.
\end{example}

\section{Mixed Semimartingale--Rough Processes}\label{sec:mixed}

\subsection{Setup}

Consider a mixed process
\begin{equation}\label{eq:mixed-process}
X_t = \alpha B_t + \beta B^H_t, \quad t \in [0, T],
\end{equation}
where $B$ is standard Brownian motion, $B^H$ is fractional Brownian motion with $H \in (0, \frac{1}{2})$, $\alpha, \beta > 0$, and $B, B^H$ are independent.

\subsection{Energy space structure}

The energy space for $X$ is the direct sum:
\[
\HX = \mathcal{H}_B \oplus \HH,
\]
where $\mathcal{H}_B$ is the Cameron--Martin space of Brownian motion (absolutely continuous functions $h$ with $h(0) = 0$ and $\norm{h}_{\mathcal{H}_B}^2 = \int_0^T |h'(t)|^2 \, dt$) and $\HH$ is the Cameron--Martin space of $B^H$. The inner product is
\[
\ip{(u,v)}{(u',v')}_{\HX} := \ip{u}{u'}_{\mathcal{H}_B} + \ip{v}{v'}_{\HH}.
\]

The stochastic integral is
\[
\delta_X(u, v) := \alpha \, \delta_B(u) + \beta \, \delta_H(v),
\]
where $\delta_B$ is the It\^o integral and $\delta_H$ is the Skorokhod integral.

\subsection{Factorization for mixed processes}

\begin{theorem}[Mixed process factorization]\label{thm:mixed}
For the mixed process $X = \alpha B + \beta B^H$ with $H \in (0, \frac{1}{2})$, the operator factorization
\[
(\Id - \E) = \delta_X \Pi_X D_X
\]
holds on $\Lp{2}(\Omega, \FT^X)$, where:
\begin{itemize}
\item $D_X F = (D_B F, D_{B^H} F) \in \Lp{2}(\Omega; \mathcal{H}_B \oplus \HH)$,
\item $\Pi_X(u, v) = (\Pi_B u, \Pi_{B^H} v)$ acts componentwise.
\end{itemize}
\end{theorem}

\begin{proof}
By independence of $B$ and $B^H$, we can apply the factorization separately to each component. For $F \in \Lp{2}(\Omega, \FT^X)$, write
\[
F - \E[F] = \underbrace{(F - \E[F \mid \mathcal{F}_T^{B^H}])}_{\text{``$B$-innovation''}} + \underbrace{(\E[F \mid \mathcal{F}_T^{B^H}] - \E[F])}_{\text{``$B^H$-innovation''}}.
\]
The first term, conditional on $\mathcal{F}_T^{B^H}$, is a centered functional of $B$ alone, so by the standard It\^o representation:
\[
F - \E[F \mid \mathcal{F}_T^{B^H}] = \delta_B(\Pi_B D_B(F - \E[F \mid \mathcal{F}_T^{B^H}])).
\]
The second term is $\mathcal{F}_T^{B^H}$-measurable and centered, so by Theorem~\ref{thm:factorization-fBM}:
\[
\E[F \mid \mathcal{F}_T^{B^H}] - \E[F] = \delta_H(\Pi_{B^H} D_{B^H}(\E[F \mid \mathcal{F}_T^{B^H}])).
\]
Combining and using that $D_B$ acts only on the $B$-component and $D_{B^H}$ only on the $B^H$-component (by independence), we obtain the componentwise factorization.
\end{proof}

\begin{remark}[Applications]\label{rem:applications}
Mixed processes of the form \eqref{eq:mixed-process} appear in rough volatility models \cite{BFG16}. Theorem~\ref{thm:mixed} provides a unified calculus treating smooth and rough components on equal footing.
\end{remark}

\section{Conclusion}

We have developed an operator-theoretic formulation of stochastic calculus for rough fractional Brownian motion based on adjointness in the energy space. The key results are:

\begin{enumerate}
\item The operator factorization $(\Id - \E) = \delta_X \Pi_X D_X$ (Theorems~\ref{thm:general-factorization} and \ref{thm:factorization-fBM}).

\item Explicit formulas for $D_X$ on cylindrical functionals (Theorem~\ref{thm:explicit-formula}).

\item Controlled martingale expansions with remainder estimates (Theorem~\ref{thm:controlled-expansion}).

\item Identification of the Gubinelli derivative as $\Pi_X D_X F$ (Theorem~\ref{thm:gubinelli-identification}).

\item Extension to mixed processes (Theorem~\ref{thm:mixed}).
\end{enumerate}

The approach clarifies the geometric origin of roughness: it arises from the interplay between adjointness ($D_X = \delta_X^*$) and adaptation ($\Pi_X$). The Gubinelli derivative is not an external construction but the intrinsic geometric object arising from these principles.

\subsection{Future directions}

\begin{itemize}
\item \textbf{Non-Gaussian processes:} Extend to L\'evy processes with rough paths.
\item \textbf{Higher-order theory:} Develop $D_X^2$ and connections to iterated integrals.
\item \textbf{Rough SDEs:} $\Lp{2}$ existence theory using operator methods.
\item \textbf{Regularity structures:} Connections to Hairer's theory \cite{Hairer14}.
\end{itemize}

\section*{Acknowledgments}

AI tools were used as interactive assistants for drafting and mathematical exploration. The author is responsible for all content.


\end{document}